\documentclass{amsart}
\usepackage{amssymb, amsmath, amsthm, graphics, comment, xspace, enumerate}
\newtheorem{thm}{Theorem}[section]

\newtheorem{lem}[thm]{Lemma}
\newtheorem{prop}[thm]{Proposition}
\theoremstyle{definition}

\theoremstyle{remark}
\newtheorem{rem}[thm]{Remark}
\numberwithin{equation}{section}

\begin{document}
\title[On the study of the wave equation set on a singular cylindrical...]{On the study of the wave equation set on a singular cylindrical
domain}

\author{Belkacem Chaouchi}
\address{Lab. de l'Energie et des Syst\' emes Intelligents, Khemis Miliana University, 44225 Khemis Miliana, Algeria}
\email{chaouchicukm@gmail.com}

\author{Marko Kosti\' c}
\address{Faculty of Technical Sciences,
University of Novi Sad,
Trg D. Obradovi\' ca 6, 21125 Novi Sad, Serbia}
\email{marco.s@verat.net}

{\renewcommand{\thefootnote}{} \footnote{
\\ \text{  }  \ \    {\it Key words and phrases.} Operational differential equation,
H\"{o}lder space, cuspidal point, wave equation.
\\  \text{  }  \ \ The second named author is partially supported by grant 174024 of Ministry
of Science and Technological Development, Republic of Serbia.}}

\begin{abstract}
In this work we give new regularity results of solutions for the linear wave
equation set in a nonsmooth cylindrical domain. Different types of conditions are imposed on the boundary of the singular domain.
Our study is performed in some particular anisotropic H\"{o}lder spaces.
\end{abstract}

\maketitle

\section{Introduction and position of the problem}

Set
\begin{equation}
\Omega :=\left \{ \left( x,y\right) \in 
\mathbb{R}
^{2}:0<x<a\text{, }\varphi _{2}\left( x\right) <y<\varphi _{1}\left(
x\right) \right \} ,  \label{domaine_cuspide}
\end{equation}%
where $a>0$ is a finite real number and $\varphi _{1},\ \varphi _{2}$ are
continuous real-valued functions defined on $\left[ 0,a\right] $ satisfying
the following conditions:
\begin{enumerate}
\item $\varphi _{1},$ $\varphi _{2}$ are of class $C^{2}$ on $\left[ 0,a%
\right] ,$

\item $\varphi :=\varphi _{1}-\varphi _{2}>0$ on $\left] 0,a\right] ,$

\item for $i\in \left \{ 1,2\right \} :\varphi _{i}\left( 0\right) =\varphi
_{i}^{\prime }\left( 0\right) =0,$

\item$\varphi _{1},$ $\varphi _{2}$ are 
strictly monotone functions.
\end{enumerate}

We suppose that the boundary $\partial \Omega $ of the cusp domain (\ref%
{domaine_cuspide}) is given by 
\begin{equation*}
\partial \Omega =\Gamma _{1}\cup \Gamma _{2}\cup \Gamma _{3},
\end{equation*}%
where%
\begin{eqnarray*}
\Gamma _{1} &:=&\left \{ \left( x,\varphi _{1}\left( x\right) \right)
:0 \leq  x \leq a\right \} , \\
\Gamma _{2} &:=&\left \{ \left( x,\varphi _{2}\left( x\right) \right)
:0 \leq x \leq a\right \} , \\
\Gamma _{3} &:=&\left \{ \left( 0,y\right) :\varphi _{2}\left( 0\right)
\leq y\leq \varphi _{1}\left( 0\right) \right \} \cup \left \{ \left( a,y\right) :\varphi _{2}\left( a\right)
\leq y\leq \varphi _{1}\left( a\right) \right \} .
\end{eqnarray*}%
In the cylinder $\Pi =\left[ 0,1\right] \times \Omega $, we consider the Cauchy
problem for the linear wave equation\textbf{\ } 
\begin{equation}
\square _{\lambda }u=h,\text{ \ }\lambda >0,  \label{LeProbleme}
\end{equation}%
equpped with the initial conditions 
\begin{equation}
\left. \mathcal{L}u\right \vert _{\left \{ 0\right \} \times \Omega }=0,
\label{V1}
\end{equation}%
\begin{equation}
\left. \mathcal{L}u\right \vert _{\left \{ 1\right \} \times \Omega }=0.
\label{V2}
\end{equation}%
Here,%
\begin{equation*}
\square _{\lambda }u:=\partial _{t}^{2}u-\partial _{x}^{2}u-\partial
_{y}^{2}u-\lambda u
\end{equation*}%
and%
\begin{equation*}
\mathcal{L}u:=\partial _{t}^{2}u+\partial _{t}u+u.
\end{equation*}%
In addition, we impose the following
boundary conditions to the problem (\ref{LeProbleme}):
\begin{equation}
\left. \partial _{y}u-u\right \vert _{\left[ 0,1\right] \times \Gamma
_{2}}=0,\left. \partial u\right \vert _{\left[ 0,1\right] \times \Gamma
_{2}}=0,  \label{neumann}
\end{equation}%
and%
\begin{equation}
\left. \partial _{x}u-u\right \vert _{\left[ 0,1\right] \times \Gamma
_{3}}=0,\left. u\right \vert _{\left[ 0,1\right] \times \left( \Gamma
_{1}\cup \Gamma _{2}\right) }=0.  \label{dirichlet}
\end{equation}%
Questions concerning the solvability of hyperbolic problems posed on
nonsmooth cylindrical domains have been studied by several authors. We cite
particularly \cite{hun1}-\cite{hun2}, where the $L^{p}$-theory of such
problems has been discussed for cylindrical domains containing a cusp
bases. The methods of investigation are derived from the well known \textit{%
a priori} estimates techniques and the potential theory.\newline

In this paper, we analyze the solvability of Problem (\ref{LeProbleme})$%
\sim $(\ref{dirichlet}) in the case that the right-hand side $h$\ belongs to the
anisotropic H\"{o}lder space $C^{\theta }\left( \left[ 0,1\right]
;L^{p}\left( \Omega \right) \right) $ with $0<\theta <1$\ and $1<p<\infty ,$
endowed with the norm 
\begin{equation*}
\left \Vert f\right \Vert _{C^{\theta }\left( \left[ 0,1\right] ;L^{p}\left(
\Omega \right) \right) }:=\sup_{t,t'\in [0,1], t\neq t'}\frac{\left \Vert f\left( t\right) -f\left(
t^{\prime }\right) \right \Vert _{L^{p}\left( \Omega \right) }}{\left \vert
t-t^{\prime }\right \vert ^{\theta }}.
\end{equation*}%
The boundary conditions (\ref{V1})$\sim $(\ref{dirichlet}) 
involve the second derivative with respect to the time variable and Robin
type conditions with respect to the space variables. Moreover, in our study, we consider
the anisotropic character of the functional framework.
The main novelty of this paper is a presentation of a new alternative abstract method for the study
of (\ref{LeProbleme})$\sim $(\ref{dirichlet}). The main idea of this method
is to transform our concrete problem to an abstract differential equation
in an appropriately chosen Banach space. The use of such an argument is
very effective and provide some interesting results concerning the
maximal regularity of solutions for Problem (\ref{LeProbleme})$\sim $(\ref%
{dirichlet}) near the singular part of the boundary of the cylindrical
domain $\Pi $. For more details about this abstract point of view, we refer
the reader to \cite{berooug1} and \cite{chaouchi}-\cite{punj}, where some elliptic and parabolic problems
on particular cusp domains have been successfully studied.

The paper is organized as follows. In the next section, we show that our
problem can be transformed by a suitable changes of variables into a
particular abstract second order differential equation. Section 3 is devoted to the
complete study of the abstract version of the transformed problem. In
Section 4, we come back to the initial problem in the cusp domain and prove
our main result. That is:

\begin{thm}
\label{main result}Let $h\in C^{\theta }\left( \left[ 0,1\right]
;L^{p}\left( \Omega \right) \right) $ with $0<\theta <1$\ and $1<p<\infty .$%
\ Then, Problem (\ref{LeProbleme})$\sim $(\ref{dirichlet}) has a unique
solution $u$ such that 
\begin{equation*}
\varphi ^{-2}u\in C^{2}\left( \left[ 0,1\right] ;L^{p}\left( \Omega \right)
\right) \ \ \mbox{ and }\ \
\partial _{t}^{2}u,\text{ }\partial _{x}^{2}u,\text{ }\partial _{y}^{2}\in
C^{\theta }\left( \left[ 0,1\right] ;L^{p}\left( \Omega \right) \right) .%
\end{equation*}
\end{thm}

\section{Change of variables and the abstract setting of the problem}

Consider the following change of variables%
\begin{equation*}
\begin{array}{ll}
T & :\Pi \rightarrow \Sigma , \\ 
& \left( t,x,y\right) \mapsto \left( t,\xi ,\eta \right) :=\left( t,-\int
\limits_{a}^{x}\dfrac{d\sigma }{\varphi \left( \sigma \right) },\dfrac{%
y-\varphi _{2}\left( x\right) }{\varphi \left( x\right) }\right) ,%
\end{array}%
\end{equation*}%
with $\Sigma $ being the semi-infinite domain given by 
$
\Sigma :=\left[ 0,1\right] \times Q,
$
where
$
Q:=\left] 0,+\infty \right] \times \left] 0,1\right[ .  \label{strip Q}
$
Now, define the following change of functions%
\begin{align*}
\left \{ 
\begin{array}{l}
v\left( t,\xi ,\eta \right) :=(v\circ T)(t,x,y)=u\left( t,x,y\right) , \\ 
\\ 
g\left( t,\xi ,\eta \right) :=(g\circ T)(t,x,y)=h\left( t,x,y\right).
\end{array}%
\right.  \label{ChangementVetG}
\end{align*}%
We have
\begin{eqnarray*}
&&\partial _{x}^{2}u+\partial _{y}^{2}u \\
&=&\partial _{\xi }^{2}v+\partial _{\eta }^{2}v+\eta \left( \varphi
_{2}^{\prime }+\eta \varphi ^{\prime }\right) ^{2}\partial _{\eta
}^{2}v-2\left( \varphi _{2}^{\prime }+\eta \varphi ^{\prime }\right)
\partial _{\eta \xi }^{2}v \\
&&+\varphi ^{\prime }\partial _{\xi }v-\left( 2\varphi ^{\prime }\varphi
_{2}^{\prime }-\varphi \varphi _{2}^{\prime \prime }-\eta \left( \varphi
\varphi ^{\prime \prime }-2\left( \varphi ^{\prime }\right) ^{2}\right)
\right) \partial _{\eta }v.
\end{eqnarray*}%
To avoid the use of weighted function spaces, we introduce the following
change of function
$
\varrho w=v,
$
where%
\begin{equation*}
\varrho =\varphi ^{2/q}\text{ and }q=\frac{p}{p-1}.
\end{equation*}%
Then we have 
\begin{eqnarray*}
&&\partial _{\xi }^{2}v+\partial _{\eta }^{2}v \\
&=&\partial _{\xi }^{2}w+\partial _{\eta }^{2}w+\frac{2}{q}\left( \frac{2}{q}%
\left( \varphi ^{\prime }\right) ^{2}+\varphi \varphi ^{\prime \prime
}\right) w+\left( \varphi _{2}^{\prime }+\eta \varphi ^{\prime }\right)
^{2}\partial _{\eta }^{2}w \\
&&-\frac{4}{q}\varphi ^{\prime }\partial _{\xi }w-2\left( \varphi
_{2}^{\prime }+\eta \varphi ^{\prime }\right) \left( \partial _{\eta \xi
}^{2}w-\frac{2}{q}\varphi ^{\prime }\partial _{\eta }w\right) +\varphi
^{\prime }\left( \partial _{\xi }w-\frac{2}{q}\varphi ^{\prime }w\right) \\
&&+\left( 2\varphi ^{\prime }\varphi _{2}^{\prime }-\varphi \varphi
_{2}^{\prime \prime }-\eta \left( \varphi \varphi ^{\prime \prime }-2\left(
\varphi ^{\prime }\right) ^{2}\right) \right) \partial _{\eta }w.
\end{eqnarray*}%
Consequently, Problem (\ref{LeProbleme}) becomes%
\begin{equation}
\begin{array}{ll}
\varrho \partial _{t}^{2}w-\partial _{\xi }^{2}w-\partial _{\eta
}^{2}w-\lambda w-\mathcal{P}w=f, & \text{in }\Sigma ,%
\end{array}
\label{Tran_prob}
\end{equation}%
where 
$
f=\varphi ^{2/q}g
$
and $\mathcal{P}$ is the second order differential operator with $C^{\infty
} $-bounded coefficients in $\Sigma ,$ given by%
\begin{eqnarray*}
&&\mathcal{P}w  \notag \\
&:=&\frac{2}{q}\left( \frac{2}{q}\left( \varphi ^{\prime }\right)
^{2}+\varphi \varphi ^{\prime \prime }\right) w+\left( \varphi _{2}^{\prime
}+\eta \varphi ^{\prime }\right) ^{2}\partial _{\eta }^{2}w
\label{Operteurpertubation} \\
&&-\frac{4}{q}\varphi ^{\prime }\partial _{\xi }w-2\left( \varphi
_{2}^{\prime }+\eta \varphi ^{\prime }\right) \left( \partial _{\eta \xi
}^{2}w-\frac{2}{q}\varphi ^{\prime }\partial _{\eta }w\right) +\varphi
^{\prime }\left( \partial _{\xi }w-\frac{2}{q}\varphi ^{\prime }w\right) 
\notag \\
&&+\left( 2\varphi ^{\prime }\varphi _{2}^{\prime }-\varphi \varphi
_{2}^{\prime \prime }-\eta \left( \varphi \varphi ^{\prime \prime }-2\left(
\varphi ^{\prime }\right) ^{2}\right) \right) \partial _{\eta }w.  \notag
\end{eqnarray*}%
It is easy to see that the conditions (\ref{neumann}) imply that%
\begin{equation*}
\left. \partial _{\xi }w-w\right \vert _{\xi =0}=0,\text{ }\left. w\right
\vert _{\xi =+\infty }=0,
\end{equation*}%
as well as that the conditions (\ref{dirichlet}) lead to%
\begin{equation*}
\left. \partial _{\eta }w-w\right \vert _{\eta =0}=0,\text{ }\left. w\right
\vert _{\eta =1}=0. \label{imagedirichlet}
\end{equation*}%
Taking into account the conditions imposed on the functions of
parametrization, the problem \eqref{Tran_prob} can be viewed as a certain perturbation of
the following one:
\begin{equation}
\begin{array}{ll}
\varrho \partial _{t}^{2}w-\partial _{\xi }^{2}w-\partial _{\eta
}^{2}w-\lambda w=f, & \text{in }\Sigma ,%
\end{array}
\label{PrincipalProblem}
\end{equation}%
accompanied with the following conditions:%
\begin{equation}
\left. \mathcal{L}w\right \vert _{t=0}=0,\text{ }\left. \mathcal{L}w\right
\vert _{t=1}=0,  \label{New IC}
\end{equation}%
and%
\begin{equation}
\begin{array}{l}
\left. \partial _{\xi }w-w\right \vert _{\xi =0}=0,\text{ }\left. w\right
\vert _{\xi =+\infty }=0, \\ 
\left. \partial _{\eta }w-w\right \vert _{\eta =0}=0,\text{ }\left. w\right
\vert _{\eta =1}=0.%
\end{array}
\label{New BC}
\end{equation}%

In what follows, we will focus our attention to the study of the principal problem (\ref%
{PrincipalProblem})$\sim $(\ref{New BC}). First of all, we will present some information
about the regularity of a new hand term $f:$

\begin{lem}
\label{effet_chang_domaine}Let $0<\mathit{\theta }<1$and $1<p<\infty .$ Then%
\begin{equation*}
h\in C^{\theta }\left( \left[ 0,1\right] ;L^{p}\left( \Omega \right) \right)
\Leftrightarrow f\in C^{\theta }\left( \left[ 0,1\right] ;L^{p}\left(
Q\right) \right) .
\end{equation*}
\end{lem}

\begin{proof}
The result is easily obtained by reiterating the same techniques used in
Proposition 3.1 in \cite{chaouchi} and Section 2.2 in \cite{guidotti}.
\end{proof}

Now, let us write the abstract version of (\ref{PrincipalProblem})$\sim $(%
\ref{New BC}). Define the following vector-valued
functions:%
\begin{eqnarray*}
w &:&\left[ 0,1\right] \rightarrow E;\text{ }t\longrightarrow w(t);\text{%
\quad }w(t)(\eta ,\upsilon )=w(t,\xi ,\eta ), \\
f &:&\left[ 0,1\right] \rightarrow E;\text{ }\xi \longrightarrow f(t);\quad
f(t)(\xi ,\eta )=f(t,\xi ,\eta ),
\end{eqnarray*}%
where $E:=L^{p}\left( Q\right) $. Problem (\ref{PrincipalProblem})$\sim 
$(\ref{New BC}) can be reduced to the abstract differential equation%
\begin{equation*}
\begin{array}{l}
\varrho w^{\prime \prime }\left( t\right) -Aw\left( t\right) -\lambda
w\left( t\right) =f\left( t\right) ,\text{ \  \  \ }0\leq t\leq 1,%
\end{array}%
\end{equation*}%
where the operator
\begin{equation}
\begin{array}{ll}
Aw & :=\partial _{\xi }^{2}w+\partial _{\eta }^{2}w
\end{array}
\label{A operator}
\end{equation}%
acts with its natural domain $D\left( A\right) $ consisting of all functions $w\in W^{2,p}( Q) $ for which
$\partial_{\xi }w-w |_{\xi =0}=0,$ $w _{\xi =+\infty}=0,$ $\partial _{\eta }w-w| _{\eta =0}=0$ and $w
|_{\eta =1}=0.$
To make the notation less cluttered, we study the following problem 
\begin{equation}
\begin{array}{l}
w^{\prime \prime }\left( t\right) -Aw\left( t\right) -\lambda w\left(
t\right) =f\left( t\right) ,\text{ \  \  \ }0\leq t\leq 1,%
\end{array}
\label{abstract equation}
\end{equation}%
accompanied with the following boundary conditions:
\begin{equation}
\left. \mathcal{L}w\right \vert _{t=0}=0,\text{ }\left. \mathcal{L}w\right
\vert _{t=1}=0. \label{Like  ventcel conditions}
\end{equation}

\begin{rem}
In the remaining part of this work, the letter $C$ will denote a generic positive
constant not necessarily the same at each occurrence.
\end{rem}

\section{On the study of the abstract problem}

In the semi-infinite strip $Q$ defined by (\ref{strip Q}), we consider the
following spectral problem%
\begin{equation}
\begin{array}{l}
\partial _{\xi }^{2}\upsilon +\partial _{\eta }^{2}\upsilon -\lambda
\upsilon =k,
\end{array}
\label{second spectral problem}
\end{equation}%
\begin{equation}
\left. \partial _{\eta }\upsilon -\upsilon \right \vert _{\eta =0}=0,\text{ }%
\left. \upsilon \right \vert _{\eta =1}=0,  \label{first bc}
\end{equation}%
\begin{equation}
\left. \partial _{\xi }\upsilon -\upsilon \right \vert _{\xi =0}=0,\text{ }%
\left. \upsilon \right \vert _{\xi =+\infty }=0,  \label{seconf bc}
\end{equation}%
where $k\in L^{p}\left( Q\right) $ and $\lambda $ is a positive spectral
parameter. It is important here to note that the spectral analysis of the
linear operator defined by (\ref{A operator}) is based essentially on the
study of problem (\ref{second spectral problem})$\sim $(\ref{seconf bc}). To
this end, we use the sum's operator theory developed in \cite{Daprato}.

\subsection{On the sum's operator theory}

Let 
$E$ a complex Banach space and $M$, $N$ be two closed linear operators with
domains $D(A)$, $D(N)$. Let $S$ be the operator defined by 
\begin{equation}
\left \{ 
\begin{array}{l}
Su:=Mu+Nu-\lambda u, \\ 
u\in D(S):=D(M)\cap D(N),%
\end{array}%
\right.  \label{sumsoperator}
\end{equation}%
where $M$ and $N$ verify the assumptions%
\begin{equation*}
(H.1)\left \{ 
\begin{array}{l}
i)\text{ }\rho (M)\supset \sum_{M}=\left \{ \mu :\left \vert \mu \right
\vert \geq r\text{, }\left \vert Arg(\mu )\right \vert <\pi -\epsilon
_{M}\right \} \text{,} \\ 
\\ 
\text{ \  \ }\forall \mu \in \sum_{M}\text{ \  \ }\left \Vert \left( M-\mu
I\right) ^{-1}\right \Vert _{L(E)}\leq C/\left \vert \mu \right \vert , \\ 
\\ 
ii)\text{ }\rho (N)\supset \sum_{N}=\left \{ \mu :\left \vert \mu \right
\vert \geq r\text{, }\left \vert Arg(\mu )\right \vert <\pi -\epsilon
_{N}\right \} , \\ 
\\ 
\text{ \  \ }\forall \mu \in \sum_{N}\text{ \  \ }\left \Vert \left( N-\mu
I\right) ^{-1}\right \Vert _{L(E)}\leq C/\left \vert \mu \right \vert , \\ 
\\ 
iii)\text{ }\epsilon _{M}+\epsilon _{N}<\pi ,\\ 
\\ 
iv)\text{ }\overline{D(M)+D(N)}=E,%
\end{array}%
\right.
\end{equation*}%
and%
\begin{equation*}
(H.2)\left \{ 
\begin{array}{l}
\forall \mu _{1}\in \rho (M),\forall \mu _{2}\in \rho (N): \\ 
\\ 
\left( M-\mu _{1}I\right) ^{-1}\left( N-\mu _{2}I\right) ^{-1}-\left( N-\mu
_{2}I\right) ^{-1}\left( M-\mu _{1}I\right) ^{-1} \\ 
\\ 
=\left[ \left( M-\mu _{1}I\right) ^{-1}\text{; }\left( N-\mu _{2}I\right)
^{-1}\right] =0\text{,}%
\end{array}%
\right.
\end{equation*}%
with $\rho (M)$ and $\rho (N)$ being the resolvent sets of $M$ and $N$ and $L(E)$ being the space of all linear continuous operators from $E$ into $E.$ The
main result proved in \cite{Daprato} is given by the following 

\begin{thm}
Assume $(H.1)$-$(H.2)$. Then the operator $S$ defined by (\ref%
{sumsoperator}) is invertible and one has 
\begin{equation}
S^{-1}:u\rightarrow -\frac{1}{2i\pi }\int_{\Gamma }\left( M+z\right)
^{-1}\left( N-\lambda -z\right) ^{-1}u\, dz,  \label{Inverse form}
\end{equation}%
where $\Gamma $ is a suitable sectorial curve lying in $\rho (M)\cap \rho
(-N)$.
\end{thm}

\subsection{On the study of Laplace operator on unbounded strip}

With a little abuse of notation, the abstract version of (\ref{second
spectral problem}) is formulated as follows%
\begin{equation}
\begin{array}{ll}
H\upsilon +B\upsilon -\lambda \upsilon =k, &  \\ 
\upsilon \in D\left( B\right) \cap D\left( H\right) , & 
\end{array}
\label{the second abstract problem}
\end{equation}%
where%
\begin{equation}
\left \{ 
\begin{array}{ll}
\left( H\upsilon \right) \left( \eta \right) & :=\upsilon ^{\prime \prime
}\left( \eta \right) , \\ 
D\left( H\right) & :=\left \{ \upsilon \in W^{2,p}\left( 0,1\right) :\upsilon
^{\prime }\left( 0\right) -\upsilon \left( 0\right) =0,\upsilon \left(
1\right) =0\right \} 
\end{array}%
\right.  \label{H Operator}
\end{equation}%
and%
\begin{equation}
\left \{ 
\begin{array}{ll}
\left( B\upsilon \right) \left( \xi \right) & :=\upsilon ^{\prime \prime
}\left( \xi \right) , \\ 
D\left( B\right) & :=\left \{ \upsilon \in W^{2,p}\left( 
\mathbb{R}
^{+}\right) :\upsilon ^{\prime }\left( 0\right) -\upsilon \left( 0\right)
=0,\upsilon \left( +\infty \right) =0\right \}.
\end{array}%
\right.  \label{B Operator}
\end{equation}%
First of all, let us observe the following:

\begin{lem}
Let $H$ be the linear operator defined by (\ref{H Operator}). Then there exists 
$C>0$ such that%
\begin{equation}
(0,\infty) \subset \rho (H)\text{ and }\left \Vert (H-\mu
I)^{-1}\right \Vert _{L(E)}\leqslant \dfrac{C}{\mu },\quad \mu>0. \label{Estimate on H}
\end{equation}
\end{lem}

\begin{proof}
As in \cite{Hou}, a direct computation shows that 
\begin{equation*}
(H-\mu I)^{-1}\upsilon =\int \limits_{0}^{+\infty }G_{1,\sqrt{\mu }}(\eta
,s)\upsilon \left( s\right) ds,
\end{equation*}%
where $\upsilon \in L^{p}\left( 0,1\right) $ and%
\begin{equation*}
G_{1,\sqrt{\mu }}(\xi ,s)=\left \{ 
\begin{array}{l}
\dfrac{\sinh \sqrt{\mu }\left( 1-\eta \right) \left[ \sinh \sqrt{\mu }s+%
\sqrt{\mu }\cosh \sqrt{\mu }s\right] }{\sqrt{\mu }\left[ \sinh \sqrt{\mu }+%
\sqrt{\mu }\cosh \sqrt{\mu }\right] },\text{\quad }0\leqslant s\leqslant \eta
, \\ 
\\ 
\dfrac{\sinh \sqrt{\mu }\left( 1-s\right) \left[ \sinh \sqrt{\mu }\eta +%
\sqrt{\mu }\cosh \sqrt{\mu }\eta \right] }{\sqrt{\mu }\left[ \sinh \sqrt{\mu 
}+\sqrt{\mu }\cosh \sqrt{\mu }\right] },\text{\quad }\eta \leqslant
s\leqslant 1.%
\end{array}%
\right.
\end{equation*}%
Here, $\sqrt{\mu }$ is the analytic determination defined by $\Re\sqrt{%
\mu }>0.$ Observe that 
\begin{eqnarray*}
&&\left \vert \sinh \sqrt{\mu }+\sqrt{\mu }\cosh \sqrt{\mu }\right \vert \\
&=&\left \vert \frac{e^{\Re\sqrt{\mu }}}{2}\left( a_{\sqrt{\mu }}+ib_{%
\sqrt{\mu }}\right) +\frac{e^{-\Re\sqrt{\mu }}}{2}\left( c_{\sqrt{\mu }%
}+id_{\sqrt{\mu }}\right) \right \vert,
\end{eqnarray*}%
where%
\begin{equation*}
\left \{ 
\begin{array}{l}
a_{\sqrt{\mu }}=1+\Re\sqrt{\mu }\cos \Im\sqrt{\mu }-\Im%
\sqrt{\mu }\sin \Im\sqrt{\mu } \\ 
b_{\sqrt{\mu }}=1+\Re\sqrt{\mu }\sin \Im\sqrt{\mu }+\Im%
\sqrt{\mu }\cos \Im\sqrt{\mu } \\ 
c_{\sqrt{\mu }}=\left( \Re\sqrt{\mu }-1\right) \cos \Im\sqrt{\mu 
}-\Im\sqrt{\mu }\sin \Im\sqrt{\mu } \\ 
d_{\sqrt{\mu }}=\left( 1-\Re\sqrt{\mu }\right) \sin \Im\sqrt{\mu 
}+\Im\sqrt{\mu }\cos \Im\sqrt{\mu }.%
\end{array}%
\right.
\end{equation*}%
Then, we obtain
\begin{eqnarray*}
&&\left \vert \sinh \sqrt{\mu }+\sqrt{\mu }\cosh \sqrt{\mu }\right \vert \\
&\geq &\frac{e^{\Re\sqrt{\mu }}}{2}\left[ \left( 1+\Re\sqrt{\mu }%
\right) ^{2}+\left( \Im\sqrt{\mu }\right) ^{2}\right] ^{1/2}-\frac{e^{-%
\Re\sqrt{\mu }}}{2}\left[ \left( 1-\Re\sqrt{\mu }\right)
^{2}+\left( \Im\sqrt{\mu }\right) ^{2}\right] ^{1/2} \\
&\geq &\sinh \Re\sqrt{\mu }\left[ 1+\left( \Re\sqrt{\mu }\right)
^{2}+2\Re\sqrt{\mu }\right] ^{1/2}
\end{eqnarray*}%
and
\begin{equation*}
\left \vert \sinh \sqrt{\mu }+\sqrt{\mu }\cosh \sqrt{\mu }\right \vert \geq
\sinh \Re\sqrt{\mu }\left[ 1+\left( \Re\sqrt{\mu }\right) \right]
\end{equation*}%
To obtain the desired result, it suffices to see that%
\begin{eqnarray*}
&&\dfrac{\sinh \left( 1-\eta \right) \Re\sqrt{\mu }}{\sinh \Re%
\sqrt{\mu }\left[ 1+\left( \Re\sqrt{\mu }\right) \right] } \\
&=&\frac{e^{\Re\sqrt{\mu }\left( 1-\eta \right) }-e^{-\Re\sqrt{%
\mu }\left( 1-\eta \right) }}{\left( e^{\Re\sqrt{\mu }}-e^{-\Re%
\sqrt{\mu }}\right) \left( 1+\left( \Re\sqrt{\mu }\right) \right) } \\
&=&\frac{e^{-\Re\sqrt{\mu }\eta }\left[ 1-e^{-\Re\sqrt{\mu }%
\left( 2-\eta \right) }\right] }{\left( 1-e^{-2\Re\sqrt{\mu }}\right)
\left( 1+\left( \Re\sqrt{\mu }\right) \right) } \\
&\leq &\frac{Ce^{-\Re\sqrt{\mu }\eta }}{\left( 1+\left( \Re\sqrt{%
\mu }\right) \right) },
\end{eqnarray*}%
from which we deduce that%
\begin{eqnarray*}
\int \limits_{0}^{1}\left \vert \dfrac{\sinh \left( 1-\eta \right) \Re%
\sqrt{\mu }}{\sinh \Re\sqrt{\mu }\left[ 1+\left( \Re\sqrt{\mu }%
\right) \right] }\right \vert ^{p}\, d\eta &\leq &\int \limits_{0}^{1}\left \vert
e^{-\Re\sqrt{\mu }\eta }\right \vert ^{p}\, d\eta \leq \frac{1}{p\Re\sqrt{\mu }}.
\end{eqnarray*}
\end{proof}

Arguing as before,  we may conclude that the following result holds true.

\begin{lem}
Let $B$ the linear operator defined by (\ref{B Operator}) Then, there exists 
$C>0$ such that%
\begin{equation}
(0,\infty)\subset \rho (B)\text{ and }\left \Vert (B-\mu
I)^{-1}\right \Vert _{L(E)}\leqslant \dfrac{C}{\mu },\quad \mu>0.  \label{Estimate on B}
\end{equation}
\end{lem}

\begin{proof}
Let $\upsilon \in L^{p}\left( 
(0,\infty)\right) $. We have
\begin{equation*}
(B-\mu I)^{-1}\upsilon =\int \limits_{0}^{+\infty }G_{2,\sqrt{\mu }}(\xi
,s)\upsilon \left( s\right) ds,
\end{equation*}%
where%
\begin{equation*}
G_{2,\sqrt{\mu }}(\eta ,s)=\left \{ 
\begin{array}{l}
\dfrac{1}{2\sqrt{\mu }}e^{-\sqrt{\mu }s}\left[ e^{\sqrt{\mu }\xi }+c
\sqrt{\mu } e^{-\sqrt{\mu }\xi }\right], \text{\quad }0\leqslant
s\leqslant \xi , \\ 
\\ 
\dfrac{1}{2\sqrt{\mu }}e^{-\sqrt{\mu }\xi }\left[ e^{\sqrt{\mu }s}+c
\sqrt{\mu }e^{-\sqrt{\mu }s}\right], \text{\quad }\xi \leqslant
s\leqslant +\infty ,%
\end{array}%
\right.
\end{equation*}%
and 
$
c\left( \sqrt{\mu }\right) =(\sqrt{\mu }-1)/(\sqrt{\mu }+1).
$
The estimate (\ref{Estimate on B}) is handled by using the same argument
delivered in \cite{Dore}, p. 1916.
\end{proof}

The following lemma is needed to justify the use of the commutative version
of the sum's operator theory:

\begin{lem}
Let $B$ and $H$ be the linear operators defined by (\ref%
{H Operator})-(\ref{B Operator}). Then
\begin{equation*}
\left \{ 
\begin{array}{l}
\forall \mu _{1}\in \rho (B),\text{ }\forall \mu _{2}\in \rho (H) \\ 
(B-\mu _{1}I)^{-1}(H-\mu _{2}I)^{-1}-(H-\mu _{2}I)^{-1}(B-\mu _{1}I)^{-1}=0.%
\end{array}%
\right.
\end{equation*}
\end{lem}

\begin{proof}
Let $\upsilon \in E$. Then the required equality follows from the following computation: 
\begin{eqnarray*}
&& (B-\mu _{1}I)^{-1} (H-\mu _{2}I)^{-1}\upsilon 
(\xi ) \\
&=&\int_{0}^{+\infty }G_{1,\sqrt{\mu_{1} }}(\xi ,s)(H-\mu _{2}I)^{-1}\upsilon (\xi)
\left( s \right) ds \\
&=&\int_{0}^{+\infty }G_{1,\sqrt{\mu_{1} }}(\xi ,s)\int_{0}^{1}G_{2,\sqrt{\mu_{2} }%
}(\xi ,s)\left[ \upsilon (s)\right] \left( \tau \right) d\tau \, ds \\
&=&\int_{0}^{1}G_{2,\sqrt{\mu_{2} }}(\xi ,s)\left( \int_{0}^{+\infty }G_{1,\sqrt{%
\mu_{1} }}(\xi ,s)\left[\upsilon (s)\right] \left( \tau \right) ds\right) d\tau
\\
&=&\int_{0}^{1}G_{2,\sqrt{\mu_{2} }}(\xi ,s)\left( \int_{0}^{+\infty }G_{1,\sqrt{%
\mu_{1} }}(\xi ,s)\upsilon (s)\, ds\right) \left( \tau \right) d\tau \\
&=& (H-\mu _{2}I)^{-1}(B-\mu _{1}I)^{-1}\upsilon
( \xi ) .
\end{eqnarray*}
\end{proof}

We conclude this section with the following useful observation:

\begin{rem}
It is necessary to note that:

\begin{enumerate}
\item The use of a classical argument of analytic continuation of the
resolvent allows us to say that the previous estimates hold true in the
sectors 
\begin{equation*}
\sum \nolimits_{B}=\left \{ \mu :\left \vert \mu \right \vert \geq r\text{, }%
\left \vert Arg(\mu )\right \vert <\pi -\epsilon _{M}\right \} 
\end{equation*}%
and%
\begin{equation*}
\sum \nolimits_{H}=\left \{ \mu :\left \vert \mu \right \vert \geq r\text{, }%
\left \vert Arg(\mu )\right \vert <\pi -\epsilon _{H}\right \} ,
\end{equation*}%
provided that $\epsilon _{M}+\epsilon _{N}<\pi .$

\item Thanks to Proposition 2.1.1 in \cite{haase}, we have that $B$ and $H$
are densely defined.
\end{enumerate}
\end{rem}

Define a closed operator $A$ by $
Aw  :=\partial _{\xi }^{2}w+\partial _{\eta }^{2}w,$ where $D(A)$ is consisted of all  
functions $w\in W^{2,p}\left( Q\right)$ such that $ \partial
_{\xi }w-w|  _{\xi =0}=0,$ $w| _{\xi =+\infty
}=0,$ $\partial _{\eta }w-w| _{\eta =0}=0$ and $w|_{\eta =1}=0.$

\begin{lem}
Let $A$ be defined as above.
Then there exists $C>0$ such that%
\begin{equation}
(0,\infty) \subset \rho (A)\mbox{ and }\left \Vert (A-\lambda
I)^{-1}\right \Vert _{L(E)}\leqslant \dfrac{C}{\lambda },\quad \lambda>0.
\label{Estimate on A}
\end{equation}
\end{lem}

\begin{proof}
This result is a direct consequence of the commutative version of the sum's
operator theory, showing that $(A-\lambda I)^{-1}$ is well defined.
The estimate (\ref{Estimate on A}) is easily handled from the formula
(\ref{Inverse form}).
\end{proof}

\subsection{Some regularity results for the abstract problem}

In order to give a fairly complete study of the abstract Cauchy problem (\ref{abstract
equation})$\sim $(\ref{Like ventcel conditions}), we follow the method
developed in \cite{acq}. We build the natural representation of the
solution by using the operational calculus and the Dunford integral. Towards this
end, consider the following scalar problem 
\begin{equation}
\begin{array}{l}
w^{\prime \prime }\left( t\right) -zw\left( t\right) =f\left( t\right) ,%
\text{ \  \  \ }0\leq t\leq 1,%
\end{array}
\label{scalar problem}
\end{equation}%
equipped with the following boundary conditions:
\begin{equation}
w^{\prime \prime }\left( 0\right) +w^{\prime }\left( 0\right) +w\left(
0\right) =0
\ \ \mbox{ and }\ \ 
w^{\prime \prime }\left( 1\right) +w^{\prime }\left( 1\right) +w\left(
1\right) =0.  \label{ICS2}
\end{equation}%
A direct computation shows that the unique solution of (\ref{scalar problem}%
)-(\ref{ICS2}) is given by 
\begin{align*}
w&\left( t\right) 
=\dfrac{1}{2\left( 1-e^{-2\sqrt{z}}\right) \sqrt{z}}\int%
\limits_{0}^{1}e^{-\sqrt{z}\left( 2-s+t\right) }f(s)\, ds
\\ & +\dfrac{1}{2\left(
1-e^{-2\sqrt{z}}\right) \sqrt{z}}\int \limits_{0}^{1}e^{-\sqrt{z}\left(
s-t+2\right) }f(s)\, ds \\
&-\frac{\left( z+\sqrt{z}+1\right) }{\left( z-\sqrt{z}+1\right) }\int
\limits_{0}^{1}e^{-\sqrt{z}\left( s+t\right) }f(s)\, ds-\frac{\left( z-\sqrt{z}%
+1\right) }{\left( z+\sqrt{z}+1\right) }\int \limits_{0}^{1}e^{-\sqrt{z}%
\left( 2-s-t\right) }f(s)\, ds \\
&+\frac{1}{2\sqrt{z}}\int \limits_{0}^{t}e^{-\sqrt{z}\left( t-s\right)
}f(s)\, ds+\frac{1}{2\sqrt{z}}\int \limits_{t}^{1}e^{-\sqrt{z}\left(
s-t\right) }f(s)\, ds.
\end{align*}%
It is well known that \eqref{Estimate on A} implies the existence of numbers $%
\delta \in \left] 0,\frac{\pi }{2}\right[ $ and $r>0$ such that the
resolvent set of $A$ contains the following sector of the complex plane

\begin{equation*}
\Pi _{\delta ,r}=\left \{ z\in 
\mathbb{C}
^{\ast }:\left \vert \arg (z)\right \vert \leqslant \delta \right \} \cup
\left \{ z\in 
\mathbb{C}
:\left \vert z\right \vert \leqslant r\right \} .  
\end{equation*}%
If $\Gamma $ denotes the sectorial boundary curve of $\Pi _{\delta ,r}$
oriented positively, then the natural representation of the solution of (\ref%
{abstract equation})$\sim $(\ref{Like ventcel conditions}), in the abstract
case, is given by 
\begin{equation}
w(t)=\sum \limits_{i=1}^{6}S_{i}\left( t,f\right) ,  \label{formal solution}
\end{equation}%
where%
\begin{equation*}
\begin{array}{l}
S_{1}\left( t,f\right) =\frac{1}{2i\pi }\int \limits_{\Gamma }\int \limits_{0}^{1}%
\dfrac{1}{2\left( 1-e^{-2\sqrt{z}}\right) \sqrt{z}}e^{-\sqrt{z}\left(
2-s+t\right) }(A-\lambda I)^{-1}f(s)\, ds, \\ 
S_{2}\left( t,f\right) =\frac{1}{2i\pi }\int \limits_{\Gamma }\int \limits_{0}^{1}%
\dfrac{1}{2\left( 1-e^{-2\sqrt{z}}\right) \sqrt{z}}e^{-\sqrt{z}\left(
s-t+2\right) }(A-\lambda I)^{-1}f(s)\, ds, \\ 
S_{3}\left( t,f\right) =-\frac{1}{2i\pi }\int \limits_{\Gamma }\int \limits_{0}^{1}%
\frac{\left( z+\sqrt{z}+1\right) }{\left( z-\sqrt{z}+1\right) }e^{-\sqrt{z}%
\left( s+t\right) }(A-\lambda I)^{-1}f(s)\, ds\, dz, \\ 
S_{4}\left( t,f\right) =-\frac{1}{2i\pi }\int \limits_{\Gamma }\int \limits_{0}^{1}%
\frac{\left( z-\sqrt{z}+1\right) }{\left( z+\sqrt{z}+1\right) }e^{-\sqrt{z}%
\left( 2-s-t\right) }(A-\lambda I)^{-1}f(s)\, ds \, dz, \\ 
S_{5}\left( t,f\right) =\dfrac{1}{2i\pi }\int \limits_{\Gamma }\int \limits_{0}^{t}%
\dfrac{e^{-\sqrt{z}\left( t-s\right) }}{2\sqrt{z}}(A-\lambda I)^{-1}f(s)\, ds,
\\ 
S_{6}\left( t,f\right) =\dfrac{1}{2i\pi }\int \limits_{\Gamma }\int \limits_{t}^{1}%
\dfrac{e^{-\sqrt{z}\left( s-t\right) }}{2\sqrt{z}}(A-\lambda I)^{-1}f(s)\, ds.%
\end{array}%
\end{equation*}%
Observe that the estimate (\ref{Estimate on A}) yields the convergence of
the integrals occurring in (\ref{formal solution}). Furthermore, we have the following proposition:

\begin{prop}
Let $f\in C^{\theta }\left( \left[ 0,1\right] ;E\right) $, with $0<\theta <1$%
\ and $1<p<\infty $. Then we have

\begin{enumerate}
\item for all $t\in \left[ 0,1\right] :S_{i}\left( t,f\right) \in D\left(
A\right) ,$ $i=1,2,...,6.$

\item for all $t\in \left[ 0,1\right] :AS_{i}\left( t,f\right) \in C^{\theta
}\left( \left[ 0,1\right] ;E\right) ,$ $i=1,2,...,6.$
\end{enumerate}
\end{prop}

\begin{proof}
The proof is carried out by analogy with the proof of Proposition 3.1 in 
\cite{berroug thesis}.
\end{proof}

As an immediate consequence of this proposition, the main result concerning
problem (\ref{abstract equation})$\sim $(\ref{Like ventcel conditions})
reads as follows:

\begin{thm}
Let $f\in C^{\theta }\left( \left[ 0,1\right] ;L^{p}\left( Q\right) \right) $
with $0<\theta <1$\ and $1<p<\infty $. Then Problem (\ref{abstract equation}%
)$\sim $(\ref{Like ventcel conditions}) has a unique solution%
\begin{equation*}
w\in C(\left[ 0,1\right] ;D\left( A\right) )\cap C^{2}(\left[ 0,1\right] ;E)
\end{equation*}%
satisfying 
$
w^{\prime \prime }\text{ and }Aw\in C^{\theta }\left( \left[ 0,1\right]
;E\right) .
$
\end{thm}

Applying all the preceding abstract results to the transformed problem 
(\ref{Tran_prob})-(\ref{New IC})-(\ref{New BC}), we obtain the
following proposition:

\begin{prop}
\label{main results for compl tans prob}\textit{Let }$f\in C^{\theta }\left( %
\left[ 0,1\right] ;L^{p}\left( Q\right) \right) ,$\ with $0<\theta <1$ and $%
1<p<+\infty $. \textit{Then, Problem (\ref{Tran_prob})-(\ref{New IC})-(\ref%
{New BC}) has a unique solution }$w\in C^{2}\left( \left[ 0,1\right]
;L^{p}\left( Q\right) \right) .$ Moreover, $w$ satisfies the maximal
regularity 
\begin{equation*}
\varrho \partial _{t}^{2}w,\text{ }\partial _{\xi }^{2}w+\partial _{\eta
}^{2}w\in C^{\theta }\left( \left[ 0,1\right] ;L^{p}\left( Q\right) \right) .%
\end{equation*}
\end{prop}

\section{Resolution of the original problem}

Recall that 
\begin{equation*}
w=\varphi ^{-2/q}v=\varphi ^{2/p}\varphi ^{-2}u.
\end{equation*}%
Proposition \ref{main results for compl tans prob} allows us to
conclude that%
\begin{equation*}
\varphi ^{-2}u\in C^{2}\left( \left[ 0,1\right] ;L^{p}\left( \Omega \right)
\right) .
\end{equation*}%
On the other hand, observe that%
\begin{equation*}
\varrho \partial _{t}^{2}w=\varphi ^{2/q}\partial _{t}^{2}w=\varphi
^{2/q}\varphi ^{-2/q}\partial _{t}^{2}v=\partial _{t}^{2}u
\end{equation*}%
and 
\begin{equation*}
\partial _{\eta }^{2}w=\varphi ^{-2/q}\varphi ^{2}\partial _{y}^{2}u=\varphi
^{2/p}\partial _{y}^{2}u.
\end{equation*}%
Keeping in mind that%
\begin{equation*}
\varrho \partial _{t}^{2}w,\text{ }\partial _{\xi }^{2}w+\partial _{\eta
}^{2}w\in C^{\theta }\left( \left[ 0,1\right] ;L^{p}\left( Q\right) \right) ,
\end{equation*}%
we may deduce that 
\begin{equation*}
\partial _{t}^{2}u,\text{ }\partial _{x}^{2}u+\partial _{y}^{2}
\end{equation*}%
belong to $C^{\theta }\left( \left[ 0,1\right] ;L^{p}\left( \Omega \right)
\right) .$This, in turn, implies that 
\begin{equation*}
\partial _{t}^{2}u,\text{ }\partial _{x}^{2}u,\text{ }\partial _{y}^{2}
\end{equation*}%
are in class $C^{\theta }\left( \left[ 0,1\right] ;L^{p}\left( \Omega \right) \right) ,
$ which completes the proof of Theorem \ref{main result}.

\end{document}